\documentclass[10pt,reqno,oneside]{amsproc}

\allowdisplaybreaks

\usepackage{amsmath, amsthm, amssymb, enumerate, enumitem, bbm,multicol}
\usepackage{color}

\chardef\forshowkeys=0
\chardef\refcheck=0
\chardef\showllabel=0
\chardef\sketches=0

\ifnum\forshowkeys=1
\definecolor{mygray}{rgb}{.6, .6, .6}

\usepackage[color,notcite,notref]{showkeys}
\usepackage{marginnote}
\fi

\usepackage{times}
\usepackage{graphicx}
\usepackage[usenames,dvipsnames,svgnames,table]{xcolor}
\usepackage[colorlinks=true, pdfstartview=FitV, linkcolor=blue, citecolor=blue, urlcolor=blue]{hyperref}

\usepackage{tikz}

\usepackage{color}
\usepackage{xcolor}

\usepackage{datetime}
\usepackage{fancyhdr,lastpage}

\ifnum\refcheck=1
\usepackage{refcheck}    
\fi

\chardef\coloryes=0 
\chardef\isitdraft=0 

\ifnum\isitdraft=1 
\def\eqref#1{({\ref{#1}})}                

\definecolor{refkey}{rgb}{.3,0.3,0.3}

\def\nnewpage{} 
\else

\def\startnewsection#1#2{\section{#1}\label{#2}\setcounter{equation}{0}}   
\def\nnewpage{} 

\fi

\begin{document}
	
	\def\ques{{\colr \underline{??????}\colb}}
	\def\nto#1{{\colC \footnote{\em \colC #1}}}
	\def\fractext#1#2{{#1}/{#2}}
	\def\fracsm#1#2{{\textstyle{\frac{#1}{#2}}}}   
	\def\nnonumber{}
	\def\les{\lesssim}
	
	\def\colr{{}}
	\def\colg{{}}
	\def\colb{{}}
	\def\colu{{}}
	\def\cole{{}}
	\def\colA{{}}
	\def\colB{{}}
	\def\colC{{}}
	\def\colD{{}}
	\def\colE{{}}
	\def\colF{{}}

	\ifnum\coloryes=1
	
	\definecolor{coloraaaa}{rgb}{0.1,0.2,0.8}
	\definecolor{colorbbbb}{rgb}{0.1,0.7,0.1}
	\definecolor{colorcccc}{rgb}{0.8,0.3,0.9}
	\definecolor{colordddd}{rgb}{0.0,.5,0.0}
	\definecolor{coloreeee}{rgb}{0.8,0.3,0.9}
	\definecolor{colorffff}{rgb}{0.8,0.9,0.9}
	\definecolor{colorgggg}{rgb}{0.5,0.0,0.4}
	\definecolor{coloroooo}{rgb}{0.45,0.0,0}
	
	\def\colb{\color{black}}
	
	\def\colr{\color{red}}
	\def\cole{\color{coloroooo}}
	
	\def\colu{\color{blue}}
	\def\colg{\color{colordddd}}
	\def\colgray{\color{colorffff}}
	
	\def\colA{\color{coloraaaa}}
	\def\colB{\color{colorbbbb}}
	\def\colC{\color{colorcccc}}
	\def\colD{\color{colordddd}}
	\def\colE{\color{coloreeee}}
	\def\colF{\color{colorffff}}
	\def\colG{\color{colorgggg}}
	
	\def\cole{}
	
	\fi
	\ifnum\isitdraft=1
	\chardef\coloryes=1 
	\baselineskip=17.6pt
	\pagestyle{myheadings}
	\def\const{\mathop{\rm const}\nolimits}  
	\def\diam{\mathop{\rm diam}\nolimits}    
	\def\dist{\mathop{\rm dist}\nolimits}    
	\def\rref#1{{\ref{#1}{\rm \tiny \fbox{\tiny #1}}}}
	\def\theequation{\fbox{\bf \thesection.\arabic{equation}}}
	\def\startnewsection#1#2{\colg \section{#1}\colb\label{#2}
		\setcounter{equation}{0}
		\pagestyle{fancy}
		\lhead{\colb Section~\ref{#2}, #1 }
		\cfoot{}
		\rfoot{\thepage\ of \pageref{LastPage}}
		
		\chead{}
		\rhead{\thepage}
		\def\nnewpage{\newpage}
		\newcounter{startcurrpage}
		\newcounter{currpage}
		\def\llll#1{{\rm\tiny\fbox{#1}}}
		\def\blackdot{{\color{red}{\hskip-.0truecm\rule[-1mm]{4mm}{4mm}\hskip.2truecm}}\hskip-.3truecm}
		\def\bluedot{{\colC {\hskip-.0truecm\rule[-1mm]{4mm}{4mm}\hskip.2truecm}}\hskip-.3truecm}
		\def\purpledot{{\colA{\rule[0mm]{4mm}{4mm}}\colb}}
		\def\pdot{\purpledot}
		\else
		\baselineskip=12.8pt
		\def\blackdot{{\color{red}{\hskip-.0truecm\rule[-1mm]{4mm}{4mm}\hskip.2truecm}}\hskip-.3truecm}
		\def\purpledot{{\rule[-3mm]{8mm}{8mm}}}
		\def\pdot{}
		\fi

		\def\textand{\qquad \text{and}\qquad}
		\def\pp{p}
		\def\qq{{\tilde p}}
		\def\KK{K}
		\def\MM{M}
		\def\ema#1{{#1}}
		\def\emb#1{#1}
		
		\ifnum\isitdraft=1
		\def\llabel#1{\nonumber}
		\else
		\def\llabel#1{\nonumber}
		\def\llabel#1{\label{#1}}
		\fi
		
		\def\tepsilon{\tilde\epsilon}
		\def\epsilonz{\epsilon_0}
		\def\restr{\bigm|}
		\def\into{\int_{\Omega}}
		\def\intu{\int_{\Gamma_1}}
		\def\intl{\int_{\Gamma_0}}
		\def\tpar{\tilde\partial}
		\def\bpar{\,|\nabla_2|}
		\def\barpar{\bar\partial}
		\def\FF{F}
		\def\gdot{{\color{green}{\hskip-.0truecm\rule[-1mm]{4mm}{4mm}\hskip.2truecm}}\hskip-.3truecm}
		\def\bdot{{\color{blue}{\hskip-.0truecm\rule[-1mm]{4mm}{4mm}\hskip.2truecm}}\hskip-.3truecm}
		\def\cydot{{\color{cyan} {\hskip-.0truecm\rule[-1mm]{4mm}{4mm}\hskip.2truecm}}\hskip-.3truecm}
		\def\rdot{{\color{red} {\hskip-.0truecm\rule[-1mm]{4mm}{4mm}\hskip.2truecm}}\hskip-.3truecm}
		
		\def\tdot{\fbox{\fbox{\bf\color{blue}\tiny I'm here; \today \ \currenttime}}}
		\def\nts#1{{\color{red}\hbox{\bf ~#1~}}} 
		
		\def\ntsr#1{\vskip.0truecm{\color{red}\hbox{\bf ~#1~}}\vskip0truecm} 
		
		\def\ntsf#1{\footnote{\hbox{\bf ~#1~}}} 
		\def\ntsf#1{\footnote{\color{red}\hbox{\bf ~#1~}}} 
		\def\bigline#1{~\\\hskip2truecm~~~~{#1}{#1}{#1}{#1}{#1}{#1}{#1}{#1}{#1}{#1}{#1}{#1}{#1}{#1}{#1}{#1}{#1}{#1}{#1}{#1}{#1}\\}
		\def\biglineb{\bigline{$\downarrow\,$ $\downarrow\,$}}
		\def\biglinem{\bigline{---}}
		\def\biglinee{\bigline{$\uparrow\,$ $\uparrow\,$}}
		\def\ceil#1{\lceil #1 \rceil}
		\def\gdot{{\color{green}{\hskip-.0truecm\rule[-1mm]{4mm}{4mm}\hskip.2truecm}}\hskip-.3truecm}
		\def\bluedot{{\color{blue} {\hskip-.0truecm\rule[-1mm]{4mm}{4mm}\hskip.2truecm}}\hskip-.3truecm}
		\def\rdot{{\color{red} {\hskip-.0truecm\rule[-1mm]{4mm}{4mm}\hskip.2truecm}}\hskip-.3truecm}
		\def\dbar{\bar{\partial}}
		\newtheorem{Theorem}{Theorem}[section]
		\newtheorem{Corollary}[Theorem]{Corollary}
		\newtheorem{Proposition}[Theorem]{Proposition}
		\newtheorem{Lemma}[Theorem]{Lemma}
		\newtheorem{Remark}[Theorem]{Remark}
		\newtheorem{definition}{Definition}[section]
		\def\theequation{\thesection.\arabic{equation}}
		\def\cmi#1{{\color{red}IK: #1}}
		\def\cmj#1{{\color{red}IK: #1}}
		\def\cml{\rm \colr Linfeng:~} 
		\def\TT{\mathbf{T}}
		\def\XX{\mathbf{X}}

		\def\sqrtg{\sqrt{g}}
		\def\DD{{\mathcal D}}
		\def\OO{\tilde\Omega}
		\def\EE{{\mathcal E}}
		\def\lot{{\rm l.o.t.}}                       
		\def\endproof{\hfill$\Box$\\}
		\def\square{\hfill$\Box$\\}
		\def\inon#1{\ \ \ \ \text{~~~~~~#1}}                
		\def\comma{ {\rm ,\qquad{}} }            
		\def\commaone{ {\rm ,\qquad{}} }         
			\def\bfx{\mathbf{x}}
		\def\dist{\mathop{\rm dist}\nolimits}    
		\def\ad{\mathop{\rm ad}\nolimits}    
		\def\sgn{\mathop{\rm sgn\,}\nolimits}   
		\def\Tr{\mathop{\rm Tr}\nolimits}    
		\def\dive{\mathop{\rm div}\nolimits}    
		\def\grad{\mathop{\rm grad}\nolimits}    
		\def\curl{\mathop{\rm curl}\nolimits}    
		\def\det{\mathop{\rm det}\nolimits}    
		\def\supp{\mathop{\rm supp}\nolimits}    
		\def\re{\mathop{\rm {\mathbb R}e}\nolimits}    
		\def\wb{\bar{\omega}}
		\def\Wb{\bar{W}}
		\def\indeq{\quad{}}                     
		\def\indeqtimes{\indeq\indeq\indeq\indeq\times} 
		\def\period{.}                           
		\def\semicolon{\,;}                      
		\newcommand{\cD}{\mathcal{D}}
		\newcommand{\cH}{\mathcal{H}}
		\newcommand{\imp}{\Rightarrow}
		\newcommand{\tr}{\operatorname{tr}}
		\newcommand{\vol}{\operatorname{vol}}
		\newcommand{\id}{\operatorname{id}}
		\newcommand{\p}{\parallel}
		\newcommand{\norm}[1]{\Vert#1\Vert}
		\newcommand{\abs}[1]{\vert#1\vert}
		\newcommand{\nnorm}[1]{\left\Vert#1\right\Vert}
		\newcommand{\aabs}[1]{\left\vert#1\right\vert}

		\ifnum\showllabel=1
		\def\llabel#1{\label{#1}}
		\else
		\def\llabel#1{\notag}
		\fi

	\title[The incompressible limit of the isentropic fluids in the analytic spaces]{The incompressible limit of the isentropic fluids in the analytic spaces}
	
	\author[L.~Li]{Linfeng Li}
	\address{Department of Mathematics\\
		University of Southern California\\
		Los Angeles, CA 90089}
	\email{lli265@usc.edu}
	\author[Y.~Tan]{Ying Tan}
	\address{Department of Mathematics\\
		University of Southern California\\
		Los Angeles, CA 90089}
	\email{yingtan@usc.edu}
	
	\begin{abstract}
	We consider the low Mach number limit problem of the Euler equations for isentropic fluids in the analytic spaces. 
	We prove that, given general analytic initial data,
	the solution is uniformly bounded on a time interval independent of the small parameter and the incompressible limit holding in the analytic norm. 
	The same results extend more generally to Gevrey initial data with convergence holds in a Gevrey norm. The results extend the isentropic fluids in \cite{JKL1} to more general pressure laws.
	\end{abstract}
	
	\maketitle
	
\startnewsection{Introduction}{sec01} 
In this paper we consider the incompressible limit of the compressible Euler equations for isentropic fluids in $\mathbb{R}^3$. 
First, the study of the low Mach number limit problem involves uniform bounds and existence of solution for a time interval independent of a small parameter $\epsilon>0$, which is related to the Mach number.
Second, one can justify the convergence to the solution of the limiting equations as $\epsilon$ tends to zero. 

For the non-isentropic flows with general initial data, the first existence results were obtained by M\'etivier and Schochet in \cite{MS01} for the whole space and the torus, with convergence holding in the whole space.
Their approach relies on the fact that one can obtain uniform bounds by applying some suitable operators to the equations. 
For domains with boundaries, the existence results and convergence in exterior domains were given by Alazard in \cite{A05}. 
The above-mentioned results hold in the Sobolev spaces.
In a recent work \cite{JKL1}, we studied the low Mach number limit in the analytic spaces in the whole space. 
The proof relies on the elliptic regularity for the velocity and the energy estimates for the entropy, curl, and the time derivative components of the velocity. 
In another recent work \cite{JKL2}, the zero Mach number limit in the analytic spaces for the exterior domains with analytic boundaries was addressed.
For the isentropic flows with well-prepared initial data, it is well-established that the solution of the compressible isentropic Euler equations exists for a time interval independent of the Mach number, and converges to the corresponding incompressible Euler equations as Mach number tends to zero (cf.~\cite{KM1, KM2}).
For other works on the isentropic flows, we refer the reader to \cite{D1, DG, KM2, LM, M}.

The main goal of this paper is to provide a simple and new approach to the low Mach number limit of the isentropic Euler equations in the analytic spaces for more general pressure laws.
Our first main result shows that, given analytic initial data, the solution is uniformly bounded on a time interval, which is independent of the Mach number. 
Our second main theorem shows that the solution of a slightly compressible Euler equations converges to the limiting equations as Mach number tends to zero. 
As the equation of state in \cite{JKL1} is assumed to be entire,
we extend the results for isentropic fluids to the case where the equation of state is more general. 

The proof differs to \cite{JKL1} and the major simplification is that the analytic energy estimates are performed using only spatial derivatives of the velocity. As a result, the main difficulty is in obtaining the explicit expressions for an arbitrary partial differential derivatives of the matrix $E(\epsilon u^\epsilon)$ (cf.~the formulation \eqref{EQ44}) in terms of $u^\epsilon$.
We overcome this by using the multivariate Faa di Bruno formula (cf.~ \cite{Br, CS}), which generalizes chain rule to higher-order derivatives. 
Moreover, we introduce a power series version of the multivariate Faa di Bruno formula which in turn estimate the commutator terms.

The paper is organized as follows. 
In Section~\ref{sec02}, we introduce the low Mach number limit problem for the isentropic Euler equations and then formulate the symmetrized hyperbolic system.
In Section~\ref{sec03}, we define the analytic spaces and state our main results. The proof of the first theorem relies on Lemma~\ref{L01}, which is proved in  Section~\ref{sec04} assuming two key estimates on the convection term and the singular term, Lemma~\ref{L04} and \ref{L05}.
Section~\ref{sec05} introduces the multivariate Faa di Bruno formula and and contains the proof of Lemma~\ref{L04} and \ref{L05}.

\startnewsection{Reformulation}{sec02} 
We consider the compressible Euler equations describing the motion of an inviscid, isentropic fluid in $\mathbb{R}^3$
	\begin{align}
	&
	\partial_t \rho + v \cdot \nabla \rho + \rho \nabla \cdot v = 0, 	
	\label{Euler01}
	\\
	&
	\rho (\partial_t v + v\cdot \nabla v) + \nabla P = 0,
	\label{Euler02}
	\end{align}
where $\rho = \rho(t,x)\in \mathbb{R}_+$ is the density, 
$v = v(t,x) \in \mathbb{R}^3$ is the velocity, and
$P= P(t,x) \in \mathbb{R}_+$ is the pressure. 
To close the system \eqref{Euler01}--\eqref{Euler02}, we assume that the fluid is barotropic and the equation of state is given in the form
	\begin{align*}
	\rho= R (P)
	,
	\end{align*}
where $R\colon \mathbb{R}_+ \to \mathbb{R}_+$ is a smooth function satisfying $\partial R/\partial P >0$.
For instance, in the case of an ideal gaseous fluid, 
	\begin{align}
	\rho = K P^{1/\gamma}
	,	
	\label{EQ200}
	\end{align}
where $K>0$ is some constant and $\gamma>1$ is the adiabatic exponent. 

After some rescalings and change of variables (cf. \cite{MS01, A05, JKL1}), we are led to the symmetric hyperbolic system 
	\begin{align}
	&
	a(\partial_t p^\epsilon + v^\epsilon \cdot \nabla p^\epsilon) +
	\frac{1}{\epsilon} \nabla \cdot v^\epsilon 
	= 
	0,
	\label{EQ56}
	\\
	&
	r (\partial_t v^\epsilon + v^\epsilon \cdot \nabla v^\epsilon)
	+ 
	\frac{1}{\epsilon} \nabla p^\epsilon 
	= 
	0,
	\label{EQ57}
	\end{align}
where $v^\epsilon = v^\epsilon(t,x)\in \mathbb{R}^3$ is the velocity, $p^\epsilon = p^\epsilon (t,x) \in \mathbb{R}$ is the pressure variation, $\epsilon \in \mathbb{R}_+$ is the Mach number, $a=a(\epsilon p^\epsilon)$, and $r=r(\epsilon p^\epsilon)$.
For example, using the equation of state \eqref{EQ200}, we obtain
\begin{align}
	a=\frac{1}{\gamma}
	\llabel{EQ201}
\end{align}
and
\begin{align}
	r= K (\bar{P} e^{\epsilon p^\epsilon})^{\frac{1}{\gamma} -1}
	,
	\llabel{EQ202}
\end{align}
for some positive constant $\bar{P}$ which represents the reference state.
For non-isentropic fluids in \cite{JKL1, JKL2}, the variable coefficients $a$ and $r$ are assumed to be products of two positive entire real-analytic functions.  
Here we consider more general pressure law that $a$ and $r$ are positive real-analytic functions of the form
\begin{align}
	\begin{split}
		a
		= 
		a(\epsilon p^\epsilon)
		\comma
		r
		=
		r(\epsilon p^\epsilon)
		.
		\label{EQ114}
	\end{split}
\end{align}
Thus we have obtained the symmetrized version of the compressible Euler equations for isentropic fluids in $\mathbb{R}^3$, which reads
\begin{align}
	E(\partial_t u^\epsilon + v^\epsilon \cdot \nabla u^\epsilon) 
	+ 
	\frac{1}{\epsilon} L(\partial_x) u^\epsilon 
	= 
	0,
	\label{EQ01}
\end{align}
where $u^\epsilon=(p^\epsilon,v^\epsilon)^T$ and
\begin{align}
	\begin{split}
		E(\epsilon u^\epsilon) 
		= 
		\begin{pmatrix}
			a(\epsilon p^\epsilon)  & 0 \\
			0 & r(\epsilon p^\epsilon)  I_3 \\
		\end{pmatrix}
		\comma
		L(\partial_x) =
		\begin{pmatrix}
			0 & \dive \\
			\nabla & 0\\
		\end{pmatrix}
		.
		\label{EQ44}
	\end{split}
\end{align}
From here on, we focus on the formulation \eqref{EQ01} in $\mathbb{R}^3$ with initial data $u_0^\epsilon = (p_0^\epsilon, v_0^\epsilon)^T$.

\startnewsection{The main results}{sec03}
We assume that the initial data $(p_0^\epsilon, v_0^\epsilon)$ satisfies
 	\begin{align}
 	\sum_{m=0}^\infty\sum_{|\alpha| = m}
 	\frac{\tau_0^{m}}{(m-3)!}
 	\Vert \partial^\alpha (p_0^\epsilon, v_0^\epsilon) \Vert_{L^2(\mathbb{R}^3)}
 	\leq 
 	M_0
 	,	
 	\label{EQ43}
	\end{align}
where $\tau_0$, $M_0>0$ are fixed constants.
For $\tau>0$, we define the analytic space 
	\begin{align*}
	A(\tau)	=\{u\in C^\infty(\mathbb{R}^3) \colon \Vert u \Vert_{A(\tau)}<\infty \}
	,
	\end{align*}
where 
	\begin{align}
	\Vert u \Vert_{A(\tau)}
	=
	\sum_{m=0}^\infty\sum_{|\alpha| = m}
	\frac{\tau(t)^{m}}{(m-3)!}
	\Vert \partial^\alpha u \Vert_{L^2(\mathbb{R}^3)}
	.
	\label{EQ33}
	\end{align}
We emphasize that, as opposed to \cite{JKL1, JKL2}, the norm does not contain time derivative of the velocity.
In \eqref{EQ33} and below, we use the convention that $n! = 1$ for $n \in -\mathbb{N}$.
We define the analyticity radius function as
	\begin{align}
	\tau(t) = \tau_0- Kt	
	,
	\label{EQ54}
	\end{align}
where $K\geq 1$ is a sufficiently large parameter to be determined below. We shall work on the time interval $[0, T]$ such that
	\begin{align}
	T \leq \frac{\tau_0}{2K}	
	.
	\label{EQ59}
	\end{align}
In view of \eqref{EQ54}--\eqref{EQ59}, we have $\tau_0/2 \leq \tau(t) \leq \tau_0$ for $t\in [0, T]$.

The first main theorem provides a uniform in $\epsilon$ boundedness of the solutions in the analytic norm on a time interval, which is independent of $\epsilon$.

\cole
\begin{Theorem}
	\label{T01}
	Assume that the initial data $(p_0^\epsilon, v_0^\epsilon)$ satisfy \eqref{EQ43}, where $\tau_0$, $M_0>0$. Then there exist sufficiently small constants $\epsilon_0$, $T_0>0$, depending on $\tau_0$ and $M_0$, such that the solution $(p^\epsilon, v^\epsilon)$ of \eqref{EQ01} satisfy the estimate
	\begin{align}
	\Vert (p^\epsilon, v^\epsilon)(t) \Vert_{A(\tau)}
	\leq
	M
	\comma
	0<\epsilon \leq \epsilon_0
	\comma
	t\in [0, T_0]
	,
	\end{align}
where $\tau$ is as in \eqref{EQ54} and $K$ and $M$ are sufficiently large constants depending on $M_0$.
\end{Theorem}
\colb

Our second main theorem states that the solution of \eqref{EQ01} converge to the solution of the stratified incompressible Euler equations
\begin{align}
	&
	\label{EQ110}
	r_0 (\partial_t v + v\cdot \nabla v) + \nabla \pi
	=
	0,
	\\
	&
	\label{EQ111}
	\dive v = 0,
\end{align}
as $\epsilon \to 0$, where $r_0 = r(0)$ is as in \eqref{EQ114}.
When the initial data are prepared, the convergence is proved in \cite{Sch86} for some $\pi$ such that $\nabla \pi$ belongs to $C^0 ([0,T], H^{s-1} (\Omega))$, where $\Omega\subset \mathbb{R}^3$ is a bounded domain and $s>1+3/2$.
We consider that the domain is $\mathbb{R}^3$ and the initial data is general (less constrained than prepared initial data).

\cole
\begin{Theorem}
	\label{T02}
	Assume that the initial data $(p_0^\epsilon, v_0^\epsilon)$ satisfy \eqref{EQ43} uniformly for fixed $\tau_0$, $M_0 >0$.
	Also, suppose that $v^\epsilon_0$ converges to $v_0$ in $H^3(\mathbb{R}^3)$.
	Then $(v^\epsilon, p^\epsilon)$ converge to $(v^{(\rm{inc})}, 0) \in L^\infty ([0,T_0], A(\delta))$ in $L^2 ([0,T], A(\delta))$, where $\delta \in (0, \tau_0]$ is a sufficiently small constant and
	$v^{(\rm{inc})}$ is the solution to \eqref{EQ110}--\eqref{EQ111}
	with initial data $w_0$, and $w_0$ is the unique solution of
	\begin{align*}
		&
		\dive w_0 = 0,
		\\
		&
		\curl (r_0 w_0) = \curl (r_0 v_0).
	\end{align*}
\end{Theorem}
\colb

In order to prove Theorem~\ref{T01}, we establish the following a~priori estimate of the (modified) velocity.
\cole
\begin{Lemma}
	\label{L01}
	Assume the initial data $(p_0^\epsilon, v_0^\epsilon)$ satisfy \eqref{EQ43}, where $M_0, \tau_0>0$. 
	Then there exist constants $\epsilon_0, T_0>0$,
	such that for all $\epsilon \in (0, \epsilon_0]$, the norm
	\begin{align}
		M_\epsilon(T) 
		:= 
		\sup_{t\in [0,T]} 
		\Vert (p^\epsilon(t), v^\epsilon (t)) \Vert_{A(\tau(t))}
	\end{align}
	satisfies the estimate
	\begin{align}
		M_\epsilon (t) 
		\leq
		C
		+
		Ct M_\epsilon (t)
		,
	\end{align}
	for $t\in [0,T_0]$, provided 
	\begin{align}
		K \geq Q(M_\epsilon (T)),
	\end{align} 
	where $K$ is as in \eqref{EQ54} and $Q$ is a nonnegative continuous function.
\end{Lemma}
\colb
For the proof of Theorem~\ref{T01} given Lemma~\ref{L01}, we refer the reader to \cite{A05, JKL1}. 
As a consequence of Theorem~\ref{T01},
the proof of Theorem~\ref{T02} is analogous to the one in \cite[Section~7]{JKL1} by using the interpolation inequality and the Stirling's formula, and thus we omit further details.

In the rest of the paper, the symbol $C$ denotes a generic constant depending on $M_0$ and $\tau_0$, which may vary from inequality to inequality. 
For simplicity of the notation, we omit the superscript $\epsilon$, and write $p$ and  $v$ for $p^\epsilon$ and $v^\epsilon$.

\cole
\begin{Remark} 
	{\rm
	\label{R01}
	(Boundedness of Sobolev norm)
	By \eqref{EQ43} and \cite[Theorem~1.1]{A05}, the $H^3$ norm of $(p^\epsilon, v^\epsilon)$ can be estimated by a constant on a time interval $[0,T']$, where $T'>0$ depends only on $\tau_0$ and $M_0$. 
	In particular, if $F$ is a smooth function of $u^\epsilon=(p^\epsilon, v^\epsilon)$, then there exists constant $C>0$ depending on $F$, such that 
	\begin{align*}
		\Vert F(\epsilon u^\epsilon) \Vert_{L^\infty_x}
		\leq
		C
		\comma
		t\in [0, T']
		\comma
		\epsilon \in (0,1]
		.
	\end{align*}
}
\end{Remark}
\colb
From here on, we shall work on the time interval $[0,T]$ where $T \in [0,T']$.

\cole
\begin{Remark}
	{\rm
	\label{R02}
	(low Mach number limit in a Gevrey norm)
	Assume that the initial data is Gevrey regular, we may generalize Theorem~\ref{T01} and \ref{T02} to the Gevrey norms by showing uniform boundedness of the solution and convergence in Gevrey spaces.
}
\end{Remark}
\colb
The proofs of Remark~\ref{R02} are analogous to those in \cite[Section~9]{JKL1}, and thus we omit the details.

\startnewsection{Proof of Lemma~\ref{L01}}{sec04}
In this section we prove Lemma~\ref{L01} and thus complete the proof of Theorem~\ref{T01}.
Since the matrix $E$ is symmetric positive definite, we may rewrite the equation \eqref{EQ01} as
	\begin{align}
	\partial_t u + v \cdot \nabla u + \frac{1}{\epsilon} E^{-1} L(\partial_x) u 
	= 
	0.
	\label{EQ30}
	\end{align}
Fix $m\in \mathbb{N}_0$ and $\alpha \in \mathbb{N}_0^3$ with $|\alpha| = m$.
We apply $\partial^\alpha$ to the equation \eqref{EQ30} to obtain
	\begin{align}
	\partial_t \partial^\alpha u
	+
	v\cdot \nabla \partial^\alpha u 
	+
	\frac{1}{\epsilon} E^{-1} L(\partial_x) \partial^\alpha u 
	=
	F_\alpha
	,
	\label{EQ34}
	\end{align}
where 
	\begin{align*}
	F_\alpha
	=
	[v\cdot \nabla, \partial^\alpha] u
	+
	\frac{1}{\epsilon}
	[E^{-1}L(\partial_x), \partial^\alpha]u
	.
	\end{align*}
Taking the $L^2$-inner product of \eqref{EQ34} with $\partial^\alpha u$ and using the Cauchy-Schwarz and H\"older inequalities, we obtain
	\begin{align}
	\frac{1}{2} \frac{d}{dt}	
	\Vert \partial^\alpha u \Vert_{L^2}^2
	+
	\frac{1}{\epsilon} \langle
	E^{-1} L(\partial_x) \partial^\alpha u, \partial^\alpha u
	\rangle
	\leq
	C
	\Vert \nabla v \Vert_{L^\infty_x} \Vert \partial^\alpha u \Vert_{L^2}^2
	+
	C
	\Vert F_\alpha \Vert_{L^2} \Vert \partial^\alpha u \Vert_{L^2}
	,
	\label{EQ32}
	\end{align}
where $\langle \cdot, \cdot \rangle$ denotes the scalar product in $L^2$.
The second term on the left side can be rewritten as
	\begin{align}
	\frac{1}{\epsilon}
	\langle	E^{-1} L(\partial_x) \partial^\alpha u,\partial^\alpha u \rangle
	=
	\frac{1}{\epsilon}
	\langle	L(\partial_x) E^{-1/2} \partial^\alpha u, E^{-1/2} \partial^\alpha u \rangle
	-
	\frac{1}{\epsilon}
	\langle	[L(\partial_x), E^{-1/2} ] \partial^\alpha u, E^{-1/2} \partial^\alpha u \rangle
	.
	\label{EQ31}
	\end{align}
Since the operator $L(\partial_x)$ is formally skew-adjoint, the first term on the right side of \eqref{EQ31} is cancelled out. 
For the second term on the right side of \eqref{EQ31}, using the product rule and chain rule, it suffices to estimate the lower-order terms. 
Note that the matrix $E$ depends on $\epsilon u$, the factor $1/\epsilon$ is cancelled out. 
From \eqref{EQ32}--\eqref{EQ31}, we use the H\"older inequality and Remark~\ref{R01} to get
	\begin{align*}
	\frac{d}{dt} \Vert \partial^\alpha u \Vert_{L^2}
	\leq
	C \Vert \partial^\alpha u \Vert_{L^2}
	+		
	C\Vert F_\alpha \Vert_{L^2} 
	.
	\end{align*}
Summing over $|\alpha| = m$, we arrive at
	\begin{align}
	\frac{d}{dt} 
    \sum_{|\alpha|= m}
	\Vert \partial^\alpha u \Vert_{L^2}
	\leq
	C \sum_{|\alpha|= m}
	\Vert F_\alpha \Vert_{L^2}
	+
	C \sum_{|\alpha|= m}
	 \Vert \partial^\alpha u \Vert_{L^2}
	.
	\label{EQ77}
	\end{align}
Denote the dissipative analytic norm by
\begin{align*}
	\Vert u \Vert_{\tilde{A}(\tau)}
	=
	\sum_{m=1}^\infty \sum_{|\alpha|=m}
	\frac{m\tau^{m-1}}{(m-3)!}
	\Vert \partial^\alpha u \Vert_{L^2}
	.
\end{align*}
Using the above notation and \eqref{EQ33}, the estimate \eqref{EQ77} implies that
	\begin{align}
	\begin{split}
	\frac{d}{dt} \Vert u \Vert_{A(\tau)}
	&
	=	
	\dot{\tau} \Vert u \Vert_{\tilde{A}(\tau)}
	+
	\sum_{m=0}^\infty \sum_{|\alpha|= m} 
	\frac{\tau^{m}}{(m-3)!} \frac{d}{dt} \Vert \partial^\alpha u \Vert_{L^2}
	\\
	&
	\leq
	\dot{\tau} \Vert u \Vert_{\tilde{A}(\tau)}
	+
	C \Vert u \Vert_{A(\tau)}
	+
	\mathcal{I}_1
	+
	\mathcal{I}_2
	,
	\label{EQ78}
	\end{split}
	\end{align}
where
	\begin{align}
	\mathcal{I}_1
	=	
	C
	\sum_{m=1}^\infty \sum_{|\alpha|= m} 
	\frac{\tau^{m}}{(m-3)!} 
	\Vert [v\cdot \nabla, \partial^\alpha] u \Vert_{L^2}
	,
	\label{EQ112}
	\end{align}
and
	\begin{align}
	\mathcal{I}_2
	=	
	\frac{C}{\epsilon}
	\sum_{m=1}^\infty \sum_{|\alpha|= m} 
	\frac{\tau^{m}}{(m-3)!} 
	\Vert [E^{-1}L(\partial_x), \partial^\alpha]u \Vert_{L^2}
	.
	\label{EQ113}
	\end{align}

In order to estimate $\mathcal{I}_1$, we use the following lemma, the proof of which is given in Section~\ref{sec05} below.

\cole
\begin{Lemma}
	\label{L04}
	There exists a constant $T_0>0$, depending on $\tau_0$ and $M_0$, such that
	\begin{align}
	\mathcal{I}_1
	\leq
	C\Vert u \Vert_{A(\tau)} \Vert u \Vert_{\tilde{A}(\tau)}
	\comma
	\epsilon \in (0,1]
	\comma
	t\in [0, T_0]
	,
	\end{align}
	for some constant $C>0$.
\end{Lemma}
\colb

The following lemma, the proof of which is given in Section~\ref{sec05} below, shall be used to estimate $\mathcal{I}_2$.
\cole
\begin{Lemma}
	\label{L05}
	There exist sufficiently small constants $\epsilon_0, T_0>0$, depending on $\tau_0$ and $M_0$ such that 
	\begin{align}
		\mathcal{I}_2	
		\leq
		C \Vert u \Vert_{\tilde{A}(\tau)}
		+
		C  \Vert u \Vert_{A(\tau)} \Vert u \Vert_{\tilde{A}(\tau)} 
		\comma
		\epsilon \in (0, \epsilon_0]
		\comma
		t\in [0, T_0]
		,
	\end{align}
	for some constant $C>0$.
\end{Lemma}
\colb

It follows from \eqref{EQ78}--\eqref{EQ113} and Lemmas \ref{L04} and \ref{L05} that
	\begin{align}
	\frac{d}{dt} \Vert u \Vert_{A(\tau)}
	\leq	
	\Vert u\Vert_{\tilde{A}(\tau)}
	(\dot{\tau} + C \Vert u\Vert_{A(\tau)}
	+
	C)
	+
	C \Vert u\Vert_{A(\tau)}
	\comma
	\epsilon \in (0,\epsilon_0]
	\comma
	t\in[0, T_0]
	,
	\label{EQ76}
	\end{align}
for some constant $C>0$.
Now, assume that the radius $\tau(t)$ decreases sufficiently fast so that the factor next to $\Vert u\Vert_{\tilde{A}(\tau)}$ is less than or equal to $0$, namely,
\begin{align*}
	K \geq Q(M_\epsilon (T))
	,
\end{align*} 
for some nonnegative continuous function $Q$.
Integrating \eqref{EQ76} in time from $0$ to $t$, we get
\begin{align*}
	\Vert u(t)\Vert_{A(\tau)}
	\leq
	C
	+
	CtM_\epsilon (t)
	.
\end{align*}
Thus, the proof of Lemma~\ref{L01} is concluded.

\startnewsection{Commutator estimates}{sec05}
Before we prove Lemma~\ref{L04} and \ref{L05},
we introduce the multivariate Faa di Bruno formula which shall be used throughout this section.
For multi-indices $\alpha, \beta \in \mathbb{N}_0^3$ where
$\alpha = (\alpha_1, \alpha_2, \alpha_3)$ and 
$\beta = (\beta_1, \beta_2, \beta_3)$, we introduce a linear order on $N_0^3$ by writing $\alpha \prec \beta$ provided one of the following holds:
\begin{enumerate}[label=(\roman*)]
	\item $|\alpha| < |\beta|$;
	\item $|\alpha| = |\beta|$ and $\alpha_1 < \beta_1$;
	\item $|\alpha| = |\beta|$, $\alpha_1=\beta_1, \ldots, \alpha_k = \beta_k$ and $\alpha_{k+1} < \beta_{k+1}$ for some $1 \leq k <3$.
\end{enumerate}
We write
\begin{align*}
	\bfx^\alpha = x_1^{\alpha_1} x_2^{\alpha_2} x_3^{\alpha_3}	
	,
\end{align*}
where $\bfx= (x_1, x_2, x_3) \in \mathbb{R}^3$.
With the above notations, we recall \cite[Theorem~2.1]{CS}.

\cole
\begin{Lemma}{(Multivariate Faa di Bruno)}
	\label{L02}
	Let $g\colon \mathbb{R}^3 \to \mathbb{R}$ be a scalar function, that is smooth in a neighborhood of some point $x_0\in \mathbb{R}^3$. Let
	$h\colon \mathbb{R} \to \mathbb{R}$ be a scalar function, smooth in a neighborhood of $y_0 = g(x_0)$.
	Define $f(x) = h(g(x))$. Then for any $\beta \in \mathbb{N}_0^3$ with $|\beta| \geq 1$, we have
	\begin{align}
		\partial^\beta f(x_0)
		=
		\sum_{i=1}^{|\beta|}
		h^{(i)}(y_0)
		\sum_{s=1}^{|\beta|}
		\sum_{P_s (i, \beta)} \beta! \prod_{l=1}^s \frac{( \partial^{\lambda_l}g(x_0) )^{k_l}}{k_l! \lambda_l!^{k_l}}
		,
	\end{align} 
	where
	\begin{align}
		\begin{split}
			P_s (i, \beta)	
			=
			\big\{ (k_1, \ldots, k_s; \lambda_1, \ldots, \lambda_s) \in \mathbb{N} \times \ldots \times \mathbb{N} \times \mathbb{N}_0^3 \times \ldots \times \mathbb{N}_0^3 \colon
			\\
			\mathbf{0} \prec \lambda_1 \prec \cdots \prec \lambda_s, 
			\sum_{l=1}^s k_l = i
			~~\text{and}~~
			\sum_{l=1}^s k_l \lambda_l = \beta
			\big\}
			.
			\label{EQ46}
		\end{split}
	\end{align}
\end{Lemma}
\colb

The next lemma provides a multivariate Faa di Bruno formula for power series, which is needed in the proof of Lemma~\ref{L05}.

\cole
\begin{Lemma}{(Power series Faa di Bruno)}
	\label{L03}
	Suppose
	\begin{align*}
		\phi(x) 
		=	
		\sum_{n=0}^\infty a_n x^n
		\comma
		x\in \mathbb{R}
		,
	\end{align*}
	where $a_n \in \mathbb{R}$ for all $n \in \mathbb{N}_0$, and
	\begin{align*}
		\psi(\bfx) 
		= 
		\sum_{\beta\in \mathbb{N}_0^3, |\beta| \geq 1} b_\beta \bfx^\beta
		\comma
		\bfx \in \mathbb{R}^3
		,
	\end{align*}
	where $b_\beta \in \mathbb{R}$ for $\beta\in \mathbb{N}_0^3$ with $|\beta| \geq 1$.
	Then the composition $\phi(\psi(\bfx))$ is a power series and can be written as
	\begin{align*}
		\phi(\psi(\bfx)) 
		= 
		\sum_{\beta\in \mathbb{N}_0^3} 
		c_\beta \bfx^\beta	
		\comma
		\bfx \in \mathbb{R}^3
		,
	\end{align*}
	where $c_{(0,0,0)} = a_0$ and
	\begin{align}
		c_\beta 
		=
		\sum_{i=1}^{|\beta|}
		\sum_{s=1}^{|\beta|}
		\sum_{P_s (i,\beta)} \binom{i}{k_1, \ldots, k_s} 
		a_i 
		\prod_{l=1}^s
		b_{\lambda_l}^{k_l}
		,
		\label{EQ67}
	\end{align}
	for $\beta \in \mathbb{N}_0^3$ with $|\beta| \geq 1$.
\end{Lemma}
\colb

\begin{proof}[Proof of Lemma~\ref{L03}]
	The composition $\phi(\psi(\bfx))$ is a formal power series and thus we assume that
	\begin{align}
		\phi(\psi(\bfx)) 
		= 
		\sum_{\beta\in \mathbb{N}_0^3} 
		d_\beta \bfx^\beta	
		\comma
		\bfx \in \mathbb{R}^3
		,
		\label{EQ70}
	\end{align}
	where $d_\beta\in \mathbb{R}$ with $\beta \in \mathbb{N}_0^3$. 
	For the constant term of \eqref{EQ70} we have $d_{(0,0,0)} = a_0$. 
	Fix $\beta\in \mathbb{N}_0^3$ where $|\beta| \geq 1$. 
	We apply $\partial^\beta$ to \eqref{EQ70} on both sides and evaluate at $\bfx = \mathbf{0}$, 
	obtaining
	\begin{align}
		\partial^\beta 
		(\phi(\psi(\bfx)))|_{\bfx = \mathbf{0}}	
		=
		\beta! d_\beta 
		.
		\label{EQ68}
	\end{align}
	On the other hand, using Lemma~\ref{L02}, we arrive at
	\begin{align}
		\begin{split}
			\partial^\beta
			 (\phi(\psi(\bfx)))|_{\bfx= \mathbf{0}}
			&	
			=
			\sum_{i=1}^{|\beta|}
			\sum_{s=1}^{|\beta|}
			\sum_{P_s (i, \beta)} 	\phi^{(i)}(\psi(\mathbf{0})) \beta! 
			\prod_{l=1}^s \frac{( \partial^{\lambda_l} \psi(\mathbf{0}) )^{k_l}}{k_l! \lambda_l!^{k_l}}	
			=
			\sum_{i=1}^{|\beta|}
			\sum_{s=1}^{|\beta|}
			\sum_{P_s (i, \beta)} i!	a_i  \beta! 
			\prod_{l=1}^s \frac{\lambda_l!^{k_l} b_{\lambda_l}  ^{k_l}}{k_l! \lambda_l!^{k_l}}	
			.
			\label{EQ69}
		\end{split}
	\end{align}
	Combining \eqref{EQ70}--\eqref{EQ69}, we obtain \eqref{EQ67} and thus conclude the proof of Lemma~\ref{L03}.
\end{proof}

\begin{proof}[Proof of Lemma~\ref{L04}]
Using the Leibniz rule, we obtain
	\begin{align}
	\begin{split}
	\mathcal{I}_1
	&
	\leq
	C
	\sum_{m=1}^\infty \sum_{j=1}^m \sum_{|\alpha|=m}  
	\sum_{|\beta|=j, \beta \leq \alpha}
	\binom{\alpha}{\beta}	
	\frac{\tau^m}{(m-3)!}
	\Vert \partial^\beta v \cdot \nabla \partial^{\alpha - \beta} u \Vert_{L^2}
	=
	\sum_{m=1}^\infty \sum_{j=1}^m
	I_{m,j}
	,
	\label{EQ63}
	\end{split}
	\end{align}
where
	\begin{align*}
	I_{m,j}	
	= 
	C
	\sum_{|\alpha| = m} \sum_{|\beta| =j, \beta \leq \alpha}
	\binom{\alpha}{\beta}
	\frac{\tau^m}{(m-3)!} 
	\Vert \partial^\beta v \cdot \nabla \partial^{\alpha - \beta} u \Vert_{L^2}
	.
	\end{align*}
We split the right side of \eqref{EQ63} according to low and high values of $j$. 
We claim that there exists a constant $C>0$ such that
	\begin{align}
	&
	\sum_{m=1}^\infty \sum_{j =1}^{[m/2]}
	I_{m,j}
	\leq
	C\Vert u \Vert_{A(\tau)} 
	\Vert u \Vert_{\tilde{A}(\tau)}
	\label{EQ35}
	\end{align}
and
	\begin{align}
	\sum_{m=1}^\infty \sum_{j =[m/2]+1}^{m}
	I_{m,j}
	\leq
	C \Vert u \Vert_{A(\tau)} \Vert u \Vert_{\tilde{A}(\tau)}
	.
	\label{EQ36}
	\end{align}

Proof of \eqref{EQ35}: Using H\"older and Sobolev inequalities we get
	\begin{align}
	\begin{split}
	\sum_{m=1}^\infty \sum_{j =1}^{[m/2]}
	I_{m,j}
	&
	\leq
	C
	\sum_{m=1}^\infty \sum_{j =1}^{[m/2]}
	\sum_{|\alpha| = m} \sum_{|\beta| = j, \beta \leq \alpha} \binom{\alpha}{\beta}
	\frac{\tau^m}{(m-3)!}
	\Vert \partial^\beta v \Vert_{L^2}^{1/4}
	\Vert D^2 \partial^\beta v \Vert_{L^2}^{3/4}
	\Vert \nabla \partial^{\alpha-\beta} u \Vert_{L^2}
	\\
	&
	\leq
	C \tau^{-3/2}
	\sum_{m=1}^\infty \sum_{j =1}^{[m/2]}
	 \sum_{|\alpha|=m} \sum_{|\beta|=j, \beta \leq \alpha}
	 	\left(
	 \frac{(m-j+1)\tau^{m-j}}{(m-j-2)!}
	 \Vert \nabla \partial^{\alpha-\beta} u \Vert_{L^2}
	 \right)
	 \\
	 &\indeqtimes
	\left(
	\frac{\tau^j}{(j-3)!}
	 \Vert \partial^\beta v \Vert_{L^2}
	\right)^{1/4}
	\left(
	\frac{\tau^{j+2}}{(j-1)!}
	\Vert D^2 \partial^\beta v \Vert_{L^2}
	\right)^{3/4}
	\mathcal{A}_{m,j,\alpha,\beta}
	,
	\label{EQ37}
	\end{split}
	\end{align}
where 
	\begin{align}
	\mathcal{A}_{m,j,\alpha,\beta}
	=
	\binom{\alpha}{\beta} 
	\frac{(j-3)!^{1/4} (j-1)!^{3/4} (m-j-2)!}{(m-j+1)(m-3)!}
	.
	\end{align}
Recall the combinatorial inequality
	\begin{align}
	\binom{\alpha}{\beta}
	\leq
	\binom{|\alpha|}{|\beta|}
	,
	\label{EQ40}
	\end{align}
which implies
	\begin{align}
	\begin{split}
	\mathcal{A}_{m,j,\alpha,\beta}
	&
	\leq
	\frac{m!}{j!(m-j)!} \frac{(j-3)!^{1/4} (j-1)!^{3/4} (m-j-2)!}{(m-j+1)(m-3)!}
	\leq
	\frac{Cm^3}{(m-j)^3}
	\leq
	C
	,
	\label{EQ38}
	\end{split}
	\end{align}
since $j \leq [m/2]$. 
Using
	\begin{align}
	\sum_{|\alpha| = m} \sum_{\beta \leq \alpha, |\beta| = j}
	x_{\beta} y_{\alpha-\beta}
	=
	\left(
	\sum_{|\beta| = j} x_{\beta}
	\right)	
	\left(
	\sum_{|\gamma| = m-j} y_{\gamma}
	\right)	
	.
	\label{EQ62}
	\end{align}
from \cite[Lemma~4.2]{KV}, together with \eqref{EQ37}--\eqref{EQ38} and the discrete H\"older inequality, we obtain
	\begin{align}
	\begin{split}
	\sum_{m=1}^\infty \sum_{j=1}^{[m/2]} 
	I_{m,j}
	&
	\leq
	C
	\sum_{m=1}^\infty \sum_{j=1}^{[m/2]} 
	\left( \sum_{|\beta| = j} \frac{\tau^j}{(j-3)!} \Vert \partial^\beta v \Vert_{L^2}
	\right)^{1/4}
	\left( \sum_{|\beta| = j} \frac{\tau^{j+2}}{(j-1)!} \Vert D^2 \partial^\beta v \Vert_{L^2}
	\right)^{3/4}
	\\
	&\indeqtimes
	\left( \sum_{|\gamma| = m-j} \frac{(m-j+1) \tau^{m-j}}{(m-j-2)!} 
	\Vert \partial^{\gamma} \nabla u \Vert_{L^2}
	\right)
	\\
	&
	\leq
	C\Vert u \Vert_{A(\tau)} 
	\Vert u \Vert_{\tilde{A}(\tau)}
	,
	\label{EQ42}
	\end{split}
	\end{align}
where the last inequality follows from the discrete Young inequality.

Proof of \eqref{EQ36}: We reverse the roles of $j$ and $m-j$ and proceed as in \eqref{EQ37}, obtaining
		\begin{align}
		\begin{split}
			&
			\sum_{m=1}^\infty \sum_{j =[m/2]+1}^{m}
			I_{m,j}
			\\
			&\indeq
			\leq
			C
			\sum_{m=1}^\infty \sum_{j =[m/2]+1}^{m}
			\sum_{|\alpha| = m} \sum_{|\beta| = j, \beta \leq \alpha} \binom{\alpha}{\beta}
			\frac{\tau^m}{(m-3)!}
			\Vert D^2 \nabla \partial^{\alpha-\beta} u \Vert_{L^2}^{3/4}
			\Vert \nabla \partial^{\alpha-\beta} u \Vert_{L^2}^{1/4}
			\Vert \partial^\beta v \Vert_{L^2}
			\\
			&\indeq
			\leq
			C \tau^{-3/2}
			\sum_{m=0}^\infty \sum_{j =[m/2]+1}^{m}
			\sum_{|\alpha|=m} \sum_{|\beta|=j, \beta \leq \alpha}
			\left(
			\frac{(m-j+3)\tau^{m-j+2}}{(m-j)!}
			\Vert D^2 \nabla \partial^{\alpha-\beta} u \Vert_{L^2}
			\right)^{3/4}
			\\
			&\indeqtimes
			\left(
			\frac{(m-j+1)\tau^{m-j}}{(m-j-2)!}
			\Vert \nabla \partial^{\alpha-\beta} u \Vert_{L^2}
			\right)^{1/4}
			\left(
			\frac{\tau^j}{(j-3)!}
			\Vert \partial^\beta v \Vert_{L^2}
			\right)
			\mathcal{B}_{m,j,\alpha,\beta}
			,
			\label{EQ39}
		\end{split}
	\end{align}
where 
	\begin{align*}
	\mathcal{B}_{m,j,\alpha,\beta}
	=
	\binom{\alpha}{\beta} 
	\frac{(m-j-2)!^{1/4} (m-j)!^{3/4}(j-3)! }{(m-j+3)^{3/4} (m-j+1)^{1/4} (m-3)!}
	.
	\end{align*}
Using \eqref{EQ40}, we get
	\begin{align}
	\begin{split}
	\mathcal{B}_{m,j,\alpha,\beta}
	&
	\leq
	\frac{m!}{j!(m-j)!} \frac{(m-j-2)!^{1/4} (m-j)!^{3/4} (j-3)! }{(m-j+3)^{3/4} 
	(m-j+1)^{1/4} (m-3)!}
	\leq
	\frac{Cm^3}{j^3}
	\leq
	C
	,
	\label{EQ41}
	\end{split}
	\end{align}
since $j \geq [m/2]+1$. We combine \eqref{EQ39}--\eqref{EQ41} and proceed as in \eqref{EQ42}, obtaining
	\begin{align*}
	\sum_{m=1}^\infty \sum_{j =[m/2]+1}^{m}
	I_{m,j}
	\leq
	C \Vert u \Vert_{A(\tau)} \Vert u \Vert_{\tilde{A}(\tau)}
	.
	\end{align*}

Combining \eqref{EQ63}--\eqref{EQ36}, we conclude the proof of Lemma~\ref{L04}.
\end{proof}

\begin{proof}[Proof of Lemma~\ref{L05}]
Using the Leibniz rule, we obtain
	\begin{align}
	\begin{split}
	\mathcal{I}_2
	&
	\leq
	\frac{C}{\epsilon} 
	\sum_{m=1}^\infty \sum_{j=1}^m \sum_{|\alpha| = m} \sum_{|\beta|=j, \beta \leq \alpha}
	\binom{\alpha}{\beta} \frac{\tau^m}{(m-3)!}
	\Vert \partial^\beta E^{-1} \nabla \partial^{\alpha -\beta} u \Vert_{L^2}
	.
	\label{EQ64}
	\end{split}
	\end{align}
We denote the non-trivial element of the matrix $E^{-1}(x)$ by $h(x)$.
Using Lemma~\ref{L02}, we arrive at
	\begin{align}
	\begin{split}
	\mathcal{I}_2
	&
	\leq
	C
	\sum_{m=1}^\infty \sum_{j=1}^m \sum_{i=1}^j \sum_{s=1}^j
	\sum_{|\alpha| = m} \sum_{|\beta|=j, \beta \leq \alpha} \sum_{P_s(i,\beta)}
	\binom{\alpha}{\beta}
	\frac{\epsilon^{i-1} \tau^m \beta!}{(m-3)!}
	\Vert 
	h^{(i)}(\epsilon u)
	\bigl(
	\prod_{l=1}^s
	\frac{	(\partial^{\lambda_l} p)^{k_l} }{k_l!\lambda_l!^{k_l}}
	\bigr)
	\nabla \partial^{\alpha-\beta} u
	\Vert_{L^2}
	\\
	&
	=
	\sum_{m=1}^\infty \sum_{j=1}^m
	J_{m,j}
	,
	\label{EQ47}
	\end{split}
	\end{align}	
where
	\begin{align*}
	J_{m,j}
	=
	C
	\sum_{i=1}^j
	 \sum_{s=1}^j
	\sum_{|\alpha| = m} \sum_{|\beta|=j, \beta \leq \alpha} \sum_{P_s(i,\beta)}
	\binom{\alpha}{\beta}
	\frac{\epsilon^{i-1} \tau^m \beta!}{(m-3)!}
	\Vert 
	h^{(i)} (\epsilon u)
	\bigl(
	\prod_{l=1}^s
	\frac{	(\partial^{\lambda_l} p)^{k_l} }{k_l!\lambda_l!^{k_l}}
	\bigr)
	\nabla \partial^{\alpha-\beta} u
	\Vert_{L^2}
	.
	\end{align*}
We split the far right side of \eqref{EQ47} according to the low and high values of $j$. We claim that there exists a constant $C>0$ such that
	\begin{align}
	&
	\sum_{m=1}^{\infty}
	\sum_{j=1}^{[m/2]} 
	J_{m,j}	
	\leq
	C \Vert u \Vert_{\tilde{A}(\tau)}
	+
	C \Vert u \Vert_{A(\tau)} \Vert u \Vert_{\tilde{A}(\tau)} 
	\label{EQ48}
	\end{align}
and
	\begin{align}
	\sum_{m=1}^{\infty}
	\sum_{j=[m/2]+1}^{m} 
	J_{m,j}	
	\leq
	C \Vert u \Vert_{A(\tau)} \Vert u \Vert_{\tilde{A}(\tau)} 
	.
	\label{EQ49}
	\end{align}

Proof of \eqref{EQ48}: Using H\"older and Sobolev inequalities, we arrive at
	\begin{align}
	\begin{split}
	J_{m,j}
	&
	\leq
	C
	\sum_{i=1}^j
	\sum_{s=1}^j
	\sum_{|\alpha| = m} \sum_{|\beta|=j, \beta \leq \alpha}\sum_{P_s(i,\beta)}
	C^i
	\binom{\alpha}{\beta}
	\frac{\epsilon^{i-1} \tau^m \beta!}{(m-3)!}
		\Vert \nabla \partial^{\alpha-\beta} u
	\Vert_{L^2}
	\Vert h^{(i)} (\epsilon u) \Vert_{L^\infty}
	\\
	&
	\indeqtimes
	\prod_{l=1}^s
	\frac{1}{k_l! \lambda_l!^{k_l}}
	\Vert D^2 \partial^{\lambda_l} p \Vert_{L^2}^{3k_l/4}
	 \Vert \partial^{\lambda_l} p \Vert_{L^2}^{k_l/4}
	\\
	&
	\leq
	C
	\sum_{i=1}^j
	\sum_{s=1}^j \sum_{|\alpha| = m} \sum_{|\beta|=j, \beta \leq \alpha}\sum_{P_s(i,\beta)}
	\frac{C^{i} \epsilon^{i-1}}{\eta^{i} \tau^{3i/2}}
		\left( \frac{(m-j+1)\tau^{m-j}}{(m-j-2)!}
	\Vert \nabla \partial^{\alpha - \beta} u \Vert_{L^2}
	\right)		
	\\&	\indeqtimes
	\left(
	\frac{\eta^i}{(i-3)!} \Vert h^{(i)} (\epsilon u) \Vert_{L^\infty}
	\right)
	\mathcal{C}_{m,j,i,s}
	\\&\indeqtimes
   \prod_{l=1}^s
	\left( \frac{\tau^{|\lambda_l|+2}}{(|\lambda_l|-1)!} \Vert D^{2} \partial^{\lambda_l} p \Vert_{L^2}
	\right)^{3k_l/4}
	\left( \frac{\tau^{|\lambda_l|}}{(|\lambda_l|-3)!} \Vert \partial^{\lambda_l} p \Vert_{L^2}
	\right)^{k_l/4}	
	,
	\label{EQ52}
	\end{split}
	\end{align}
where $\eta \in (0,1]$ is some constant to be determined below and
	\begin{align}
	\begin{split}
	\mathcal{C}_{m,j,i,s}
	&
	=
	\binom{\alpha}{\beta}
	\frac{(i-3)! (m-j-2)! \beta!}{(m-3)!(m-j+1)}	
	\prod_{l=1}^s \frac{(|\lambda_l|-1)!^{3k_l/4} (|\lambda_l|-3)!^{k_l/4}}{ k_l! \lambda_l!^{k_l}}
	.
	\label{EQ101}
	\end{split}
	\end{align}
For $\beta, \beta_1, \cdots, \beta_k \in \mathbb{N}_0^3$ and $k\in \mathbb{N}$,
using \eqref{EQ40} and induction, we obtain
\begin{align}
	\begin{split}
	\binom{\beta}{\beta_1, \beta_2, \cdots, \beta_k}
	\leq
	\binom{|\beta|}{|\beta_1|, |\beta_2|, \cdots, |\beta_k|}
	,
	\label{EQ100}
	\end{split}
\end{align}
where
\begin{align*}
	\beta 
	= 
	\sum_{i=1}^k \beta_i
	.
\end{align*}
From \eqref{EQ101}--\eqref{EQ100} we arrive at
	\begin{align}
	\begin{split}
		\mathcal{C}_{m,j,i,s}
		&
		\leq
		\frac{Cm!}{j!(m-j)!}
		\frac{(m-j-2)!}{(m-3)!(m-j+1)}
		\binom{i}{k_1, \ldots, k_s} 
		\frac{|\beta|! }{\prod_{l=1}^s |\lambda_l|!^{k_l}}
		\prod_{l=1}^s (|\lambda_l|-1)!^{3k_l/4} (|\lambda_l|-3)!^{k_l/4}
		\\
		&
		\leq
		 \frac{Cm^3}{(m-j)^3}
		\binom{i}{k_1, \ldots, k_s}
		\leq
		C \binom{i}{k_1, \ldots, k_s}
		,
		\label{EQ53}
	\end{split}
	\end{align}
since $j \leq [m/2]$.
Recall that $h(x)$ is an analytic function.
Thus, for any fixed $\zeta >0$, there exists a constant $\eta \in (0,1]$ such that
\begin{align}
\begin{split}
	\Vert D^i h(x) \Vert_{L^\infty} 
	\leq
	\frac{C (i-3)!}{\eta^i}
	\comma
	|x| \leq \zeta
	\comma
	i \in \mathbb{N}_0
	,
	\label{EQ102}
\end{split}	
\end{align}
for some constant $C>0$. It follows from Remark~\ref{R01} that
\begin{align}
	\begin{split}
		\Vert D^i h(\epsilon u) \Vert_{L^\infty} 
		\leq
		\frac{C (i-3)!}{\eta^i}
		\comma
		x\in \mathbb{R}^3
		\comma
		i \in \mathbb{N}_0
		,
		\label{EQ120}
	\end{split}	
\end{align}
by choosing $\epsilon \leq \zeta / CM_\epsilon (T)$.
Combining \eqref{EQ52}, \eqref{EQ53}, and \eqref{EQ120}, together with splitting the low and high values of $i$, we obtain
	\begin{align}
	\begin{split}
	&	
	\sum_{m=1}^{\infty}
	\sum_{j=1}^{[m/2]} 
	J_{m,j}	
	\\
	&\indeq
	\leq
	C
	\sum_{m=1}^{\infty}
	\sum_{j=1}^{[m/2]} \sum_{i=2}^{j}
	\sum_{s=1}^j \sum_{|\alpha| = m} \sum_{|\beta|=j, \beta \leq \alpha}\sum_{P_s(i,\beta)}
	\binom{i}{k_1, \ldots, k_s}
	\left( \frac{(m-j+1)\tau^{m-j}}{(m-j-2)!}
	\Vert \nabla \partial^{\alpha - \beta} u \Vert_{L^2}
	\right)	
	\\
	&\indeqtimes
	\frac{C^{i} 	\epsilon^{i/2}}{\eta^{i} \tau^{3i/2}}
	\prod_{l=1}^s
	\left( \frac{\tau^{|\lambda_l|+2}}{(|\lambda_l|-1)!} \Vert D^{2} \partial^{\lambda_l} p \Vert_{L^2}
	\right)^{3k_l/4}
	\left( \frac{\tau^{|\lambda_l|}}{(|\lambda_l|-3)!} \Vert \partial^{\lambda_l} p \Vert_{L^2}
	\right)^{k_l/4}
	\\
	&\indeq\indeq
	+
	\frac{C}{\eta \tau^{3/2}}\sum_{m=1}^{\infty}
	\sum_{j=1}^{[m/2]} 
	 \sum_{|\alpha| = m} \sum_{|\beta|=j, \beta \leq \alpha}
	\left( \frac{(m-j+1)\tau^{m-j}}{(m-j-2)!}
	\Vert \nabla \partial^{\alpha - \beta} u \Vert_{L^2}
	\right)	
	\\
	&\indeqtimes
	\left( \frac{\tau^{j+2}}{(j-1)!} \Vert D^{2} \partial^{\beta} p \Vert_{L^2}
	\right)^{3/4}
	\left( \frac{\tau^{j}}{(j-3)!} \Vert \partial^{\beta} p \Vert_{L^2}
	\right)^{1/4}
	\\
	&\indeq
	=
	J_1
	+
	J_2
,
	\label{EQ71}
	\end{split}
	\end{align}
where the inequality follows from 
\[
	P_s(1,\beta) = 
	\begin{cases}
		\{(1;\beta)\},  & \text{$s=1 $;}\\
		\emptyset, & \text{$s\geq 2$.}
	\end{cases}
\]
For the term $J_1$, we set the power series
	\begin{align*}
	\phi(x) 
	= 
	\sum_{n=2}^\infty
	(C \epsilon^{1/2} \eta^{-1} \tau^{-3/2})^n
 	x^n 
	\comma
	x\in \mathbb{R}
	\end{align*}
and 
	\begin{align*}
	\psi(\bfx) 
	= 
	\sum_{\beta \in \mathbb{N}_0^3, |\beta| \geq 1} b_\beta \bfx^\beta
	\comma
	\bfx\in \mathbb{R}^3
	,
	\end{align*}
where 
	\begin{align}
	b_\beta
	=
	\left( \frac{\tau^{|\beta|+2}}{(|\beta|-1)!} \Vert D^{2} \partial^{\beta} p \Vert_{L^2}
	\right)^{3/4}
	\left( \frac{\tau^{|\beta|}}{(|\beta|-3)!} \Vert \partial^{\beta} p \Vert_{L^2}
	\right)^{1/4}
	\label{EQ81}
	\end{align}
for $\beta \in \mathbb{N}_0^3$ with $|\beta| \geq 1$. 
It follows from Lemma~\ref{L03} that
	\begin{align}
	\sum_{n=2}^\infty 
	(C \epsilon^{1/2} \eta^{-1} \tau^{-3/2})^n
	\left(
		\sum_{\beta \in \mathbb{N}_0^3, |\beta| \geq 1} b_\beta \bfx^\beta
	\right)^n
	=
	\sum_{\beta \in \mathbb{N}_0^3, |\beta| \geq 1}
	e_\beta \bfx^\beta
	\comma
	\bfx \in \mathbb{R}^3
	,
	\label{EQ72}
	\end{align}
where	
	\begin{align}
	\begin{split}
	e_\beta
	&
	=
	\sum_{i=2}^{|\beta|} \sum_{s=1}^{|\beta|} \sum_{P_s (i,\beta)}
	(C \epsilon^{1/2} \eta^{-1} \tau^{-3/2})^i
	\binom{i}{k_1, \ldots, k_s}
	\prod_{l=1}^s
	\left( \frac{\tau^{|\lambda_l|+2}}{(|\lambda_l|-1)!} \Vert D^{2} \partial^{\lambda_l} p \Vert_{L^2}
	\right)^{3k_l/4}
	\\&\indeqtimes
	\left( \frac{\tau^{|\lambda_l|}}{(|\lambda_l|-3)!} \Vert \partial^{\lambda_l} p \Vert_{L^2}
	\right)^{k_l/4}
	,
	\label{EQ80}
	\end{split}
	\end{align}
for $\beta \in \mathbb{N}_0^3$ with $|\beta| \geq 1$.
From \eqref{EQ62}, \eqref{EQ71}, and \eqref{EQ80}, we arrive at
	\begin{align}
	\begin{split}
	J_1	
	&
	\leq
	C
	\sum_{m=1}^{\infty}
	\sum_{j=1}^{[m/2]} 
	 \sum_{|\alpha| = m} \sum_{|\beta|=j, \beta \leq \alpha}
	\left( \frac{(m-j+1)\tau^{m-j}}{(m-j-2)!}
	\Vert \nabla \partial^{\alpha - \beta} u \Vert_{L^2}
	\right)	e_\beta
	\\
	&
	\leq
	C
	\sum_{m=1}^\infty \sum_{j=1}^{[m/2]}
	\bigl( \sum_{|\beta| = j} e_\beta
	\bigr)
	\left( \sum_{|\gamma| = m-j} \frac{(m-j+1)\tau^{m-j}}{(m-j-2)!}
	\Vert \nabla \partial^{\gamma} u \Vert_{L^2}
	\right)
	\\&
	\leq
	C
	\Vert u \Vert_{\tilde{A}(\tau)}
	\sum_{n=2}^\infty 
	\bigl(
	C \epsilon^{1/2} \eta^{-1} \tau^{-3/2}
	\sum_{\beta \in \mathbb{N}_0^3, |\beta| \geq 1}
	b_\beta
	\bigr)^n
	,
	\label{EQ73}
	\end{split}
	\end{align}
where the last inequality follows from \eqref{EQ72} with $\bfx = (1, 1, 1)$. 
Note that from \eqref{EQ81} and the discrete Young inequality, we have
	\begin{align}
	\sum_{\beta \in \mathbb{N}_0^3, |\beta| \geq 1}
	b_\beta
	\leq
	C \Vert u \Vert_{A(\tau)}
	.
	\label{EQ115}
	\end{align}
Choosing $\epsilon \leq \eta^2 \tau^{3}/C M_{\epsilon}(T)^2$ and combining \eqref{EQ73}--\eqref{EQ115}, we arrive at
	\begin{align}
	J_1
	\leq
	C \Vert u \Vert_{\tilde{A}(\tau)}
	.
	\label{EQ82}
	\end{align}
For the term $J_2$, we proceed as in \eqref{EQ37}--\eqref{EQ42} to get
	\begin{align}
	J_2
	\leq
	C \Vert u \Vert_{A(\tau)}	\Vert u \Vert_{\tilde{A}(\tau)}
	.
	\label{EQ83}
	\end{align}
Combining \eqref{EQ71} and \eqref{EQ82}--\eqref{EQ83}, we conclude the proof of \eqref{EQ48}.

Proof of \eqref{EQ49}:
Using H\"older and Sobolev inequalities, together with \eqref{EQ100}, we arrive at
\begin{align}
	\begin{split}
		J_{m,j}
		&
		\leq
		C
		\sum_{i=1}^j
		\sum_{s=1}^j
		\sum_{|\alpha| = m} \sum_{|\beta|=j, \beta \leq \alpha}\sum_{P_s(i,\beta)}
		C^i
		\binom{\alpha}{\beta}
		\frac{\epsilon^{i-1} \tau^m |\beta|!}{(m-3)! k_s! |\lambda_s|!^{k_s}}
		\Vert h^{(i)} (\epsilon u) \Vert_{L^\infty}
		\\
		&
		\indeqtimes
		\left(
		\prod_{l=1}^{s-1}
		\frac{1}{k_l!|\lambda_l|!^{k_l}}
		\Vert D^2 \partial^{\lambda_l} p \Vert_{L^2}^{3k_l/4}
		\Vert \partial^{\lambda_l} p \Vert_{L^2}^{k_l/4}
		\right)
		\Vert \partial^{\lambda_s} p \Vert_{L^2}
		\Vert D^2 \partial^{\lambda_s} p \Vert_{L^2}^{3(k_s-1)/4}
		\\
		&\indeqtimes
		\Vert \partial^{\lambda_s} p \Vert_{L^2}^{(k_s-1)/4}
		\Vert D^2 \nabla \partial^{\alpha-\beta} u
		\Vert_{L^2}^{3/4}
		\Vert \nabla \partial^{\alpha-\beta} u
		\Vert_{L^2}^{1/4}
		\\
		&
		\leq
		C\sum_{i=1}^j
		\sum_{s=1}^j \sum_{|\alpha| = m} \sum_{|\beta|=j, \beta \leq \alpha}\sum_{P_s(i,\beta)}
		\frac{C^{i} \epsilon^{i-1}}{\eta^{i} \tau^{3i/2}}
		\left( \frac{\eta^{i}}{(i-3)!} \Vert h^{(i)}(\epsilon u) \Vert_{L^\infty}
		\right)
		\\
		&
		\indeqtimes
		\prod_{l=1}^{s-1}
		\left( \frac{\tau^{|\lambda_l|+2}}{(|\lambda_l|-1)!} \Vert D^{2} \partial^{\lambda_l} p \Vert_{L^2}
		\right)^{3k_l/4}
		\left( \frac{\tau^{|\lambda_l|}}{(|\lambda_l|-3)!} \Vert \partial^{\lambda_l} p \Vert_{L^2}
		\right)^{k_l/4}	
		\\
		&\indeqtimes
		\left( \frac{\tau^{|\lambda_s|}}{(|\lambda_s|-3)!} 
		\Vert \partial^{\lambda_s} p \Vert_{L^2}
		\right)
		\left( \frac{\tau^{|\lambda_s|+2}}{(|\lambda_s|-1)!} \Vert D^{2} \partial^{\lambda_s} p \Vert_{L^2}
		\right)^{3(k_s-1)/4}
		\\&\indeqtimes
		\left( \frac{\tau^{|\lambda_s|}}{(|\lambda_s|-3)!} \Vert \partial^{\lambda_s} p \Vert_{L^2}
		\right)^{(k_s-1)/4}	
				\left( \frac{(m-j+3)\tau^{m-j+2}}{(m-j)!}
		\Vert D^2 \nabla \partial^{\alpha - \beta} u \Vert_{L^2}
		\right)^{3/4}	
		\\
		&
		\indeqtimes
		\left( \frac{(m-j+1)\tau^{m-j}}{(m-j-2)!}
		\Vert \nabla \partial^{\alpha - \beta} u \Vert_{L^2}
		\right)^{1/4}	
		\mathcal{D}_{m,j,i,s}
		,
		\label{EQ75}
	\end{split}
\end{align}
where
	\begin{align*}
	\begin{split}
		\mathcal{D}_{m,j,i,s}
		&
		=
		\binom{\alpha}{\beta}
		\frac{(i-3)! (m-j-2)!^{1/4} (m-j)!^{3/4} |\beta|!}{(m-3)!(m-j+1)^{1/4} (m-j+3)^{3/4}}	
				\frac{(|\lambda_s|-1)!^{3(k_s-1)/4} (|\lambda_s|-3)!^{(k_s-1)/4} (|\lambda_s|-3)!}{ k_s! |\lambda_s|!^{k_s}}
		\\&\indeqtimes
		\prod_{l=1}^{s-1} \frac{(|\lambda_l|-1)!^{3k_l/4} (|\lambda_l|-3)!^{k_l/4}}{ k_l! |\lambda_l|!^{k_l}}
		.
	\end{split}
	\end{align*}
Using \eqref{EQ40}, we arrive at
	\begin{align*}
		\mathcal{D}_{m,j,i,s}
		\leq
		C
		\binom{i}{k_1, \ldots, k_s} \frac{m! (|\lambda_s| -3)!}{(m-3)! |\lambda_s|!}
			\leq
		C \binom{i}{k_1, \ldots, k_s} \frac{m^3}{|\lambda_s|^3}
		.
	\end{align*}
Note that from \eqref{EQ46} we get
	\begin{align*}
	 |\beta|	
	 =
	\sum_{l=1}^s k_l |\lambda_l| 
	\leq
	|\lambda_s| \sum_{l=1}^s k_l
	=
	i 	|\lambda_s|
	,
	\end{align*}
from where
	\begin{align}
	\mathcal{D}_{m,j,i,s}
	\leq
	C \binom{i}{k_1, \ldots, k_s}
	\frac{m^3 i^3}{|\beta|^3}
	\leq
	C \binom{i}{k_1, \ldots, k_s} i^3
	,
	\label{EQ74}
	\end{align}
since $|\beta| = j \geq [m/2]+1$. 
From  \eqref{EQ120} and \eqref{EQ75}--\eqref{EQ74}, we obtain
	\begin{align}
	\begin{split}
	&
	\sum_{m=1}^\infty \sum_{j=[m/2]+1}^{m} 
	J_{m,j}
	\\
	&\indeq
	\leq	
	C
	\sum_{m=1}^\infty \sum_{j=[m/2]+1}^{m} \sum_{i=2}^{j}
	\sum_{s=1}^j \sum_{|\alpha| = m} \sum_{|\beta|=j, \beta \leq \alpha}\sum_{P_s(i,\beta)}
	\frac{C^{i} 	\epsilon^{i/2}}{\eta^{i} \tau^{3i/2}}
	\binom{i}{k_1, \ldots, k_s}
	\\
	&\indeqtimes
	\left( \frac{\tau^{|\lambda_s|}}{(|\lambda_s|-3)!}
	\Vert \partial^{\lambda_s} p \Vert_{L^2}
	\right)	
	\left( \frac{(m-j+1)\tau^{m-j}}{(m-j-2)!}
	\Vert \nabla \partial^{\alpha - \beta} u \Vert_{L^2}
	\right)^{1/4}
	\\
	&\indeqtimes
	\left( \frac{(m-j+3)\tau^{m-j+2}}{(m-j)!}
	\Vert D^2 \nabla \partial^{\alpha - \beta} u \Vert_{L^2}
	\right)^{3/4}	
	\\
	&\indeqtimes
	\prod_{l=1}^{s-1}
	\left( \frac{\tau^{|\lambda_l|+2}}{(|\lambda_l|-1)!} \Vert D^{2} \partial^{\lambda_l} p \Vert_{L^2}
	\right)^{3k_l/4}
	\left( \frac{\tau^{|\lambda_l|}}{(|\lambda_l|-3)!} \Vert \partial^{\lambda_l} p \Vert_{L^2}
	\right)^{k_l/4}
		\\
	&\indeqtimes
	\left( \frac{\tau^{|\lambda_s|+2}}{(|\lambda_s|-1)!} \Vert D^{2} \partial^{\lambda_s} p \Vert_{L^2}
	\right)^{3(k_s-1)/4}
	\left( \frac{\tau^{|\lambda_s|}}{(|\lambda_s|-3)!} \Vert \partial^{\lambda_s} p \Vert_{L^2}
	\right)^{(k_s-1)/4}
	\\
	&\indeq\indeq
	+
	C \sum_{m=1}^{\infty}
	\sum_{j=[m/2]+1}^{m} 
	 \sum_{|\alpha| = m} \sum_{|\beta|=j, \beta \leq \alpha}
	\left( \frac{\tau^{j}}{(j-3)!}
	\Vert \partial^{\beta} p \Vert_{L^2}
	\right)	
	\\
	&\indeqtimes
	\left( \frac{(m-j+1)\tau^{m-j}}{(m-j-2)!}
	\Vert \nabla \partial^{\alpha - \beta} u \Vert_{L^2}
	\right)^{1/4}
		\left( \frac{(m-j+3)\tau^{m-j+2}}{(m-j)!}
	\Vert D^2 \nabla \partial^{\alpha - \beta} u \Vert_{L^2}
	\right)^{3/4}
	\\
	&\indeq
	=
	J_3
	+
	J_4
	.
	\label{EQ116}
	\end{split}		
	\end{align}
For the term $J_3$, we use the change of variable $i'= i-1$ to get
\begin{align}
	\begin{split}
	J_3 
	&
	\leq	
	C
	\sum_{m=1}^\infty \sum_{j=[m/2]+1}^{m} 
	\sum_{|\alpha| = m} \sum_{|\beta|=j, \beta \leq \alpha} \sum_{\omega \leq \beta, |\omega| \geq 1}
	\sum_{i'=1}^{|\beta -\omega|}
	\sum_{s=1}^{|\beta -\omega|} 
	\sum_{P_s(i',\beta-\omega)}
	\frac{C^{i'} 	\epsilon^{(i'+1)/2}}{\eta^{i'+1} \tau^{3(i'+1)/2}}
	\binom{i'}{k'_1, \ldots, k'_s}
	\\
	&\indeqtimes
	\left( \frac{\tau^{|\omega|}}{(|\omega|-3)!}
	\Vert \partial^{\omega} p \Vert_{L^2}
	\right)	
	\left( \frac{(m-j+1)\tau^{m-j}}{(m-j-2)!}
	\Vert \nabla \partial^{\alpha - \beta} u \Vert_{L^2}
	\right)^{1/4}
	\\
	&\indeqtimes
	\left( \frac{(m-j+3)\tau^{m-j+2}}{(m-j)!}
	\Vert D^2 \nabla \partial^{\alpha - \beta} u \Vert_{L^2}
	\right)^{3/4}	
	\\
	&\indeqtimes
	\prod_{l=1}^{s}
	\left( \frac{\tau^{|\lambda_l'|+2}}{(|\lambda_l'|-1)!} \Vert D^{2} \partial^{\lambda_l'} p \Vert_{L^2}
	\right)^{3k'_l/4}
	\left( \frac{\tau^{|\lambda_l'|}}{(|\lambda_l'|-3)!} \Vert \partial^{\lambda_l'} p \Vert_{L^2}
	\right)^{k'_l/4}
	,
	\label{EQ123}
	\end{split}
\end{align}
where the element of $ P_s(i',\beta-\omega)$ is denoted by
$(k'_1\dots,k'_s; \lambda'_1,\dots\lambda'_s) $.
Fix $m,j,k\in \mathbb{N}_0$ with $k \leq j \leq m$. Using \eqref{EQ62} we get
\begin{align}
	\sum_{|\alpha| =m}
	\sum_{\beta \leq \alpha, |\beta| = j}
	\sum_{\omega \leq \beta, |\omega| = k}
	x_{\omega} y_{\alpha - \beta} z_{\beta - \omega}
	=
	\left(
	\sum_{|\gamma| = k} x_\gamma
	\right)
	\left(
	\sum_{|\gamma| = m-j} y_\gamma
	\right)
	\left(
	\sum_{|\gamma| = j-k} z_\gamma
	\right)
	.
	\label{EQ121}
\end{align}
From \eqref{EQ123}--\eqref{EQ121}, we proceed as in \eqref{EQ71}--\eqref{EQ82} to get
	\begin{align}
		\begin{split}
	J_3
	&
	\leq
	C
	\sum_{m=1}^\infty \sum_{j=[m/2]+1}^{m} \sum_{k=1}^j
	\bigl( \sum_{|\gamma| = j-k} e_\gamma
	\bigr)
	\bigl( \sum_{|\gamma| = k} 
	\frac{\tau^{k}}{(k-3)!} \Vert \partial^\gamma p\Vert_{L^2}
	\bigr)
	\\&\indeqtimes
	\left( \sum_{|\gamma| = m-j} \frac{(m-j+1)\tau^{m-j}}{(m-j-2)!}
	\Vert \nabla \partial^{\gamma} u \Vert_{L^2}
	\right)
	+
	C \epsilon \Vert u\Vert_{A(\tau)}^2 \Vert u\Vert_{\tilde{A}(\tau)}
	\\
	&
	\leq
	C
	\Vert u \Vert_{\tilde{A}(\tau)}
	\Vert u\Vert_{A(\tau)}
	\sum_{n=2}^\infty 
	\bigl(
	C \epsilon^{1/2} \eta^{-1} \tau^{-3/2}
	\sum_{\beta \in \mathbb{N}_0^3, |\beta| \geq 1}
	b_\beta
	\bigr)^n
	+
	C \epsilon \Vert u\Vert_{A(\tau)}^2 \Vert u\Vert_{\tilde{A}(\tau)}
	\\
	&
	\leq
	C 	\Vert u \Vert_{\tilde{A}(\tau)}
	\Vert u\Vert_{A(\tau)}
	,
	\label{EQ103}
	\end{split}
	\end{align}
by choosing $\epsilon \leq \eta^2 \tau^3/ CM_\epsilon (T)^2$, where $e_\gamma$ and $b_\beta$ are as in \eqref{EQ80} and \eqref{EQ81}.
For the term $J_4$, we proceed as in \eqref{EQ37}--\eqref{EQ42} to get
\begin{align}
	J_4
	\leq
	C\Vert u\Vert_{A(\tau)} \Vert u\Vert_{\tilde{A}(\tau)}
	.
	\label{EQ104}
\end{align}
Combining \eqref{EQ116}--\eqref{EQ104}, we conclude the proof of \eqref{EQ49}.

The proof of Lemma~\ref{L05} is thus completed by combining \eqref{EQ47}--\eqref{EQ49}.
\end{proof}

\section*{Acknowledgments}
LL was supported in part by the NSF grants DMS-2009458 and DMS-1907992, while YT was supported in part by the NSF grant DMS-1106853.
The authors are grateful to Juhi Jang and Igor Kukavica for fruitful discussions.

\end{document}